\documentclass[runningheads,a4paper]{llncs}
\usepackage{amsmath, epsfig,graphicx}
\usepackage{amssymb,bbm}
\usepackage{color}
\usepackage{latexsym,url}

\newcommand{\GG}{{\mathcal G}}
\newcommand{\DD}{{\mathcal D}}

\newcommand{\HH}{{\mathcal H}}

\newcommand{\oR}{{\mathbb R}}
\newcommand{\oN}{{\mathbb N}}
\newcommand{\EE}{\mathcal E}
\newcommand{\SSS}{{\mathcal S}}
\newcommand{\oZ}{{\mathbb Z}}

\newcommand{\bS}{\mathbf S}

\newcommand{\psd}{{\mathcal S}^n_+}

\newcommand{\PSD}{\mathcal{S}_{+}}
\newcommand{\EDM}{\text{\rm EDM}}

\newcommand{\fedm}{{\text{\rm ed}}}
\newcommand{\gd}{\text{\rm gd}}

\newcommand{\oG}{\overline{G}}

\newcommand{\ran}{{\text{\rm rank}}}

\newcommand{\qb}{\bold{q}}
\newcommand{\pb}{\bold{p}}
\newcommand{\p}{p}

\newcommand{\Om}{\Omega}
\newcommand{\la}{\langle}
\newcommand{\ra}{\rangle}
\newcommand{\what}{\widehat}

\newcommand{\ignore}[1]{}

\begin{document}

\title{The   Gram dimension of a graph}

\author{Monique Laurent\inst{1,2} \and Antonios Varvitsiotis\inst{1}}
\institute{Centrum Wiskunde \& Informatica (CWI), Amsterdam, The Netherlands \\ \email{\{M.Laurent, A.Varvitsiotis\}@cwi.nl} \and Tilburg University, The Netherlands}
\maketitle

\begin{abstract}
The  Gram dimension $\gd(G)$ of a graph is the smallest integer $k \ge 1$ such that, for every assignment of unit vectors to the nodes of the graph, there exists another assignment of unit vectors lying in $\oR^k$,
having the same inner products on the edges of the graph.
The class of graphs satisfying $\gd(G) \le k$ is minor closed for fixed $k$, so it can characterized by a finite list of forbidden minors.
For $k\le 3$, the only forbidden minor is $K_{k+1}$. We show that a graph has Gram dimension at most 4 if and only if it does not have $K_5$ and $K_{2,2,2}$ as minors. We also show some close connections to the notion of $d$-realizability of graphs.
In particular, our result implies the characterization of 3-realizable graphs of Belk and Connelly \cite{Belk,BC}.
\end{abstract}

\section{Introduction}

The problem of completing a given partial matrix (where only a subset of entries are specified) to a full positive semidefinite (psd) matrix is one of the most extensively studied  matrix completion problems.
A particular instance is   the completion problem for correlation matrices arising in probability and statistics, and it is also closely related to the completion problem for 
Euclidean distance matrices  with applications, e.g.,  to sensor network localization and  molecular conformation in chemistry.
We refer, e.g.,  to \cite{DL97,L01} and further references therein for additional
 details.  

An important feature of a matrix is its rank which intuitively can be seen as a 
 measure of complexity of the data it represents. 
As an  example, the minimum embedding dimension of a finite metric space can be expressed as the rank of an appropriate matrix~\cite{DL97}. Another problem of interest is to compute  low rank solutions to semidefinite programs as they may lead to  improved approximations to the underlying discrete optimization problem~\cite{AZ05}.
Consequently,   the problem of 
computing (approximate) matrix completions is  of fundamental importance in many disciplines 
and it has  been extensively studied (see, e.g.,  \cite{CR08,RFP}).

This motivates the following question which we study in this paper:
Given a partially specified matrix  which admits at least one psd completion,
 provide guarantees for the existence of small rank psd completions. 

Evidently, the (non)existence of small rank completions depends on the values of the prescribed entries of the partial  matrix.  We  approach this problem from a combinatorial point of view and  give an  answer in terms of the combinatorial structure of the problem, which is captured by the {\em Gram dimension} of the  graph. 
 Before we give the precise definition,  we  introduce  some notation.

\medskip
Throughout $\SSS^n$ denotes the set of symmetric $n\times n$ matrices and  $\SSS^n_+$ (resp., $\SSS^n_{++}$)
is the subset of all positive semidefinite (resp., positive definite) matrices.
For a matrix $X$ the notation $X\succeq 0$ means that  $X$ is psd.
Given  a graph  $G=(V=[n],E)$, its edges are denoted as (unordered) pairs $(i,j)$ and, for convenience, we will sometimes 
 identify $V$ with the set of all diagonal pairs, i.e., we set     $V=\{(i,i)\mid i\in [n]\}$.
Moreover, $\pi_{V E}$ denotes the projection from $\SSS^n$ onto the subspace $\oR^{V\cup E}$  indexed by the diagonal entries and the edges of $G$.

\begin{definition}\label{gramdef}
The Gram dimension $\gd(G)$ of a graph $G=([n],E)$ is  the smallest integer $k\ge 1$ such that,
for any matrix $X\in \SSS^n_+$,
there exists  another matrix $X' \in \SSS^n_+$ with rank at most $k$ and such that
$\pi_{VE}(X)=\pi_{VE}(X^{'})$.
\end{definition}
Given a graph $G=([n],E)$, a partial $G$-matrix  is  a partial $n \times n $ matrix whose entries are specified on the diagonal and at positions corresponding to edges of $G$.
Then,  if  a partial $G$-matrix   admits a psd completion,  it also has  one  of rank at most $\gd(G)$.
This motivates the study of bounds for  $\gd(G)$.

\medskip
As we will see in Section~\ref{basic},  the class of graphs with $\gd(G)\le k$ is closed under taking minors for any fixed $k$, hence it can be characterized  in terms of a finite list of forbidden minors. Our main result is such a characterization for $k\le 4$.

\medskip
\noindent{\bf Main Theorem.} For $k\le 3$, a  graph $G$ has $\gd(G)\le k$  if and only if it  has no $K_{k+1}$ minor. For $k=4$, a graph $G$ has $\gd(G)\le 4$ if and only if it  has no $K_5$ and $K_{2,2,2}$  minors.\medskip

An equivalent way of rephrasing the notion of Gram dimension 
 is in terms of ranks of feasible solutions to semidefinite programs.
 Indeed, the  Gram dimension of a graph $G=(V,E)$ is at most $k$ if and only if the set
 $$S(G,a)=\{ X \succeq 0 \mid X_{ij}=a_{ij} \text{ for }  ij \in V \cup E\}$$
 contains a matrix of rank at most $k$ for all $a\in \oR^{V\cup E}$ for which $S(G,a)$ is not empty.
The set $S(G,a)$ is a typical instance of spectrahedron. Recall that a  
  {\em spectrahedron} is the convex region  defined as the intersection of the positive semidefinite cone with a finite set  of linear subspaces, i.e., the feasibility region 
  of a semidefinite program  in canonical  form:
\begin{equation}\label{sdpsform}\max \la A_0,X \ra \text{ subject to } \la A_j,X\ra=b_j,\   (j=1,\ldots,m), \qquad X \succeq 0.
\end{equation}
  If  the feasibility region
  of (\ref{sdpsform}) is not empty,
it follows from well known geometric results that it contains a 
matrix $X$ of rank $k$ satisfying 
${k+1\choose 2}\le m$, that is, $k\le \lfloor \frac{\sqrt{8m+1}-1}{2}\rfloor$
 (see \cite{Bar01}). 
 Applying this to the spectrahedron $S(G,a)$,  
  we obtain the  bound
 $\gd(G)=O(\sqrt{|V|+|E|)}$, which is however weak  in general.

\medskip
As an application,  the Gram dimension can be used to bound the rank
 of optimal solutions to semidefinite programs. 
 Indeed consider  a semidefinite program 
in canonical form  (\ref{sdpsform}). Its {\em aggregated sparsity pattern} is the graph $G$  with node set $[n]$ and 
 whose edges are the pairs corresponding to the positions where at least one of the matrices  $A_j$ ($j\ge 0$) has a nonzero entry.
Then, whenever (\ref{sdpsform}) attains its maximum, it admits  an optimal solution of rank at most $\gd(G)$. 
Results ensuring existence of low rank solutions are important, in particular, for approximation algorithms.
Indeed 
semidefinite programs are widely used as convex tractable relaxations to hard combinatorial problems. Then the rank one solutions typically correspond  to the  desired  optimal solutions of the discrete problem and low rank solutions can lead to improved performance guarantees (see e.g. the result of \cite{AZ05} for max-cut).

As an illustration,
consider the max-cut problem for graph $G$ and its standard semidefinite programming relaxation:
\begin{equation}\label{maxcut}\max \frac{1}{4}\la L,X\ra \text{ subject to  }  X_{ii} =1\  (i=1,\ldots,n), \ \   X\succeq 0,
\end{equation}
where $L$ denotes the Laplacian matrix of $G$. Clearly, 
 the aggregated sparsity pattern of  program 
(\ref{maxcut}) is equal to  $G$. In particular, our main Theorem implies that if $G$ does not have $K_5$ and $K_{2,2,2}$ minors, then program (\ref{maxcut}) has an optimal
solution of rank at most four. Of course, this is not of great  interest  since for $K_5$ minor free graphs, the  max-cut problem  can be solved in polynomial time (\cite{B83}).

In a similar flavor, for a graph $G=([n],E)$ and  $w\in \oR^{V\cup E}_+$,   the problem of computing bounded rank solutions  to semidefinite programs of the form
$$\max \sum_{i=1}^n w_{i}X_{ii}\ \text{ s.t. } \sum_{i,j=1}^nw_{i} w_{j}X_{ij}=0,\
X_{ii}+X_{ij}-2X_{ij}\le w_{ij}\ (ij\in E),\ X\succeq 0,$$
 has been studied in \cite{GHW}.
In particular, it is shown in \cite{GHW} that there always exists an  optimal solution of rank at most the tree-width of $G$ plus 1.
There are numerous other results related to geometric representations of graphs; we refer, e.g., 
 to \cite{Hog08,Lo95,Lo01} for further results and references.
 
\medskip
Yet another, more geometrical, way of interpreting the Gram dimension is in terms of graph
 embeddings  in the spherical metric space. 
 For this, consider the unit sphere $\bS^{k-1}=\{x \in \oR^k |\  \|x\|=1\}$,  equipped with the distance 
$$d_{\bS}(x,y)=\arccos (x^Ty) \ \text{ for } x,y\in \bS^{k-1}.$$
Here, $\|x\|$ denotes the usual Euclidean norm.
Then $(\bS^{k-1},d_{\bS})$ is a metric space, known as  the {\em spherical metric space.}  
A graph $G=([n],E)$ has Gram dimension at most $k$ if and only if, for any assignment of vectors $p_1,\ldots,p_n \in \bS^d$ (for some $d\ge 1$),
there exists another assignment  $q_1,\ldots,q_n \in \bS^{k-1}$ such that 
$$d_{\bS}(p_i,p_j)=d_{\bS}(q_i,q_j), \text{ for } ij \in E.$$
In other words, this is the question of deciding whether a partial matrix can be embedded in the $(k-1)$-dimensional spherical space.
The analogous question for the Euclidean metric space $(\oR^k,\|\cdot \|)$ 
has been extensively studied. In particular,  Belk and Connelly \cite{Belk,BC}  show the following result 
for the graph parameter 
$\fedm(G)$, the analogue of 
$\gd(G)$ for Euclidean embeddings,  introduced in Definition \ref{defedm}. 

\begin{theorem}\label{theoBC}  For $k\le 2$,  $\fedm(G)\le k$  if and only if $G$ has no $K_{k+2}$ minor. For $k=3$,   $\fedm(G)\le 3$ if and only if $G$ does not have $K_5$ and $K_{2,2,2}$  minors. 
\end{theorem}

There is a  striking similarity between our main Theorem and Theorem~\ref{theoBC} above.  This is no coincidence, since these two parameters are very closely related as we will see in Section \ref{seclinks}.

\medskip
The paper is orgranized as follows. In Section~\ref{basic} we give definitions   and  establish some basic properties of the graph parameter $\gd(G)$.
 In   Section~\ref{gramdim4} we sketch the proof of our main Theorem. In Section~\ref{boundingviasdp} we  show how we can use semidefinite programming in order to prove that $\gd(V_8)$ and $\gd(C_5\times C_2)$ are both at most  four. In Section~\ref{seclinks}  we will elaborate between the similarities and differences between the two graph parameters $\gd(G)$ and $\fedm(G)$. Section~\ref{sec:complexity} discusses  the complexity of the natural decision problem  associated with the graph parameter $\gd(G)$. 
 
\medskip
 {\bf Note.} The extended version of this paper is available at~\cite{LV12}. Complexity issues associated with the parameter $\gd(G) $ are further discussed in~\cite{NLV12}.

\section{Basic definitions and properties}\label{basic}
For a  graph $G=(V=[n],E)$ let  $ \PSD(G)=\pi_{VE}(\psd)\subseteq \oR^{V\cup E}$ denote the projection of the positive semidefinite cone onto $\oR^{V\cup E}$,  whose elements can be seen as the partial  $G$-matrices  
 that can be completed to a psd matrix. Let $\EE_n$ denote the set of matrices in $\SSS^n_+$ with an all-ones diagonal (aka the correlation matrices), and let
$\EE(G)=\pi_{E}(\EE_n)\subseteq \oR^E$ denote its projection onto the edge subspace $\oR^E$,  known as the {\em elliptope} of $G$;
we only project on the edge set since all diagonal entries are  implicitly known and equal to one for matrices in $\EE_n$.
\begin{definition}Given a graph $G=(V,E)$ and a vector $a\in \oR^{V \cup E}$,  a Gram representation of $a$ in $\oR^k$
consists of  a set of vectors $\p_1,\ldots,\p_n\in \oR^k$  such that $$\p_i^T\p_j=a_{ij}\  \forall  ij \in V \cup E.$$
The Gram dimension of   $a\in \PSD(G)$, denoted as $\gd(G,a)$,  is  the smallest integer $k$ for which $a$ has a Gram representation in $\oR^k$. 
\end{definition}

\begin{definition} 
The Gram dimension of a graph $G=(V,E)$ is defined as
\begin{equation}\label{gramdimdef}
 \gd(G)=\underset{a \in \PSD(G)}{\max} \gd(G,a).
 \end{equation}
\end{definition}
We  denote by  $\GG_k$  the class of graphs for which $\gd(G)\le k$. 
Clearly, the maximization in (\ref{gramdimdef}) can be restricted  to be taken over  all $a \in \EE(G)$ (where all diagonal entries are implicitly taken to be equal to 1).

\medskip
We now investigate the behavior of the graph parameter $\gd(G)$ under some simple graph operations.

\begin{lemma}\label{lemminor}
The graph parameter $\gd(G)$ is monotone nondecreasing with respect to edge deletion and contraction.
That is, if $H$ is a minor of $G$ (denoted as $H\preceq G$), then $\gd(H)\le \gd(G)$.
\end{lemma}

\begin{proof} Let $G=([n],E)$ and $e\in E$. It is clear that $\gd(G\backslash e)\le \gd(G)$. We show that 
$\gd(G\slash e)\le \gd(G)$. 
Say $e$ is the edge $(1,n)$ and $G\slash e=([n-1],E')$.
  Consider $ X \in \SSS_{+}^{n-1}$; we show that there exists $X'\in\SSS^{n-1}_+$ with rank at most $k=\gd(G)$ and such that $\pi_{E'}(X)=\pi_{E'}(X')$.
  For this, extend $X$ to the matrix $Y\in\SSS^n_+$ defined by $Y_{nn}=X_{11}$ and $Y_{in}=X_{1i}$ for $i\in[n-1]$.
  By assumption, there exists $Y'\in\SSS^n_+$ with rank at most $k$ such that $\pi_E(Y)=\pi_E(Y')$. Hence
  $Y'_{1i}=Y'_{ni}$ for all $i\in [n]$, so that 
  the principal submatrix $X'$ of $Y'$ indexed by $[n-1]$ has rank at most $k$ and satisfies
  $\pi_{E'}(X')=\pi_{E'}(X)$.
  \ignore{
   where $X=\left(\begin{array}{cc}
Y & b\\
b^T& 1
\end{array}\right)$. Define   $\what{X}=\left(\begin{array}{ccc}
 Y &  b& b\\
 b^T& 1& 1\\
 b^T & 1 & 1
 \end{array}\right)$ and notice that $\what{X} \in \psd$. Then there exists a matrix $X^{'} \in \psd$ with  $\ran X^{'} \le k$ such that $\pi_{E(G)}(\what{X})=\pi_{E(G)}(X^{'})$. Let $X^{'}[n]$ be the principal minor of $X^{'}$ defined by the first $n$ rows/columns. For $ij \in E(G/e)$ where $i,j\not=n-1$ we have that $X_{i,j}^{'}=X_{i,j}$.  Moreover if  $(n-1,i) \in  E(G/e)\cap E(G)$ then $X_{n-1,i}^{'}=\what{X}_{n-1,i}=X_{n-1,i}$. Lastly, for $(n-1,i) \in E(G/e)$ where $(n,i)\in E(G)$  we have that $X_{n-1,i}^{'}=X^{'}_{n,i}=\what{X}_{n,i}=X_{n-1,i}$.
 Summarizing, we have  that  $\pi_{E(G/e)}(X^{'}[n])=\pi_{E(G/e)}(X)$ and since $\ran X^{'}[n]\le k$ the claim follows. }
\qed\end{proof}

Let  $G_1=(V_1,E_1)$, $G_2=(V_2,E_2)$ be two graphs,  where  $V_1\cap V_2$  is a clique in both $G_1$ and $G_2$. Their {\em clique sum} is the graph $G=(V_1\cup V_2, E_1\cup E_2)$, also called their 
{\em clique $k$-sum} when $k=|V_1\cap V_2|$. The following result follows from well known arguments (used already, e.g.,  in \cite{GJSW87}).

\begin{lemma}\label{lemcliquesum}
If $G$ is the clique sum of two graphs $G_1$ and $G_2$,  then 
$$\gd(G)=\max\{\gd(G_1),\gd(G_2)\}.$$
\end{lemma}

\ignore{

\begin{proof}
It suffices to show  $\gd(G)\le k:=\max \{\gd(G_1),\gd(G_2)\}$ (the other inequality follows from Lemma \ref{lemminor}).
For this let $x\in \PSD(G)$ and consider its projections $x_i\in \PSD(G_i)$ for $i=1,2$.
By assumption, $x_i$ admits a Gram representation $u_{i,h}\in \oR^k$ ($h\in V_i$) for $i=1,2$.
As $V_0=V_1\cap V_2$ induces a clique in both $G_1,G_2$, 
 we have $u_{1,h}^Tu_{1,k}=u_{2,h}^Tu_{2,k}$ for all $h,k\in V_0$. Hence  there exists an orthogonal matrix $U$ mapping $u_{1,h}$ to $u_{2,h}$ for all $h\in V_0$.
Therefore, the vectors $w_h:= Uu_{1,h}$ for $h\in V_1$ and $w_h:=u_{2,h}$ for $h\in V_2$ form a Gram representation 
 of $x$ in $\oR^k$.
\qed \end{proof}

}

As a direct application, one can bound the Gram dimension of partial $k$-trees.
Recall that a graph $G$ is a  {\em $k$-tree} if it is a clique $k$-sum of copies of $K_{k+1}$ and  a {\em partial $k$-tree} if it is a subgraph of a $k$-tree (equivalently, $G$ has tree-width $k$).
Partial 1-trees are exactly the forests and  partial 2-trees (aka series-parallel graphs) are the graphs  with no $K_4$ minor (see \cite{Du65}).

\begin{lemma}\label{ktrees} If $G$ is a partial $k$-tree then $\gd(G)\le k+1$.
\end{lemma}

For example,  for the complete graph $K_n$, $\gd(K_n)=n$, and $\gd(K_n\setminus e)=n-1$ for any edge $e$ of $K_n$. Moreover, for the complete bipartite graph $K_{n,m}$ ($n \le m$),     $\gd(K_{n,m})=n+1$ (since $K_{n,m}$ is a partial $n$-tree and contains a $K_{n+1}$ minor).

\medskip
In view of Lemma \ref{lemminor}, the class $\GG_k$ of graphs with Gram dimension at most $k$ is closed under taking minors. 
Hence, by the celebrated graph minor theorem \cite{RS}, 
it can be characterized by finitely many minimal forbidden minors. 
The simple properties we just established  suffice to characterize $\GG_k,$ for $k\le 3$.
\begin{theorem}\label{theosmallr}
For $k\le 3$, $\gd(G)\le k$  if and only if $G$ has no minor $K_{k+1}$.
\end{theorem}

The next natural question  is to characterize the graphs with Gram dimension at most 4, which we address in the next section.

\section{Characterizing graphs with Gram dimension  at most  4}\label{gramdim4}

In this section  we characterize the class  of graphs with Gram dimension at most 4.
Clearly, $K_5$ is a minimal forbidden minor for $\GG_4$.
We now  show that this is also the case for the complete tripartite graph $K_{2,2,2}$.
\ignore{
\begin{figure}[h]
\centering \includegraphics[scale=0.5]{figs/K222-ML}
\caption{The graph $K_{2,2,2}$.} 
\label{k222}
\end{figure}
}

\begin{lemma}\label{lemK222}
The graph $K_{2,2,2}$ is a minimal forbidden minor for   $\GG_4$.
\end{lemma}

\begin{proof}
First we construct $x\in \EE(K_{2,2,2})$ with $\gd(K_{2,2,2},x)=5$. For this, let $K_{2,2,2}$ be obtained from $K_6$ by deleting the edges $(1,4)$, $(2,5)$ and $(3,6)$. Let $e_1,\ldots,e_5$ denote the standard unit vectors in $\oR^5$,
 let $X$ be the Gram matrix of the vectors 
$e_1,e_2,e_3,e_4,e_5$ and $(e_1+e_2)/\sqrt 2$ labeling the nodes $1,\ldots,6$, respectively, and let 
 $x\in \EE(K_{2,2,2})$ be the projection of $X$. We now verify that $X$ is the unique psd completion  of $x$ which shows  that $\gd(K_{2,2,2})\ge 5$. Indeed the chosen Gram labeling of the matrix $X$ implies the following linear dependency:  $C_6=(C_4+C_5)/\sqrt 2$ among its columns $C_4,C_5,C_6$ indexed respectively by $4,5,6$;
this implies that the unspecified entries $X_{14}, X_{25}, X_{36}$ are uniquely determined in terms of the specified entries of $X$.

On the other hand, one can easily verify that  $K_{2,2,2}$ is a partial 4-tree, 
thus $\gd(K_{2,2,2})\le 5$. Moreover,  
deleting or contracting an edge in $K_{2,2,2}$ yields a partial 3-tree, thus with Gram dimension at most 4.
\ignore{
Next we verify that $K_{2,2,2}$ is a minimal forbidden minor for membership in $\GG_4$, i.e., 
  when deleting or contracting an edge one obtains a graph with Gram dimension at most 4.
Let $G$ be the graph obtained by deleting  the edge $(1,2)$ in $K_{2,2,2}$. If we add the edge $(3,6)$ to $G$, we obtain  
a graph which is the clique sum of three cliques on four nodes, namely the cliques on $\{3,4,5,6\}$, $\{2,3,4,6\}$, and on $\{1,3,5,6\}$. This implies that the Gram dimension of $G$ is at most 4.

Let $G$ be the graph obtained by contracting  the edge $(1,2)$ in $K_{2,2,2}$. As the subgraph of $G$ induced by $\{3,4,5,6\}$ is the clique sum of two $K_3$'s, its Gram dimension is at most 3 and thus the Gram dimension of $G$ is at most~4.
}		
\qed\end{proof}

By Lemma~\ref{ktrees} we know that all partial 3-trees belong to  $\GG_4$. Moreover, it is known that partial 3-trees can be characterized in terms of four forbidden minors as stated below. 
\begin{figure}[h] 
\centering \includegraphics[scale=0.4]{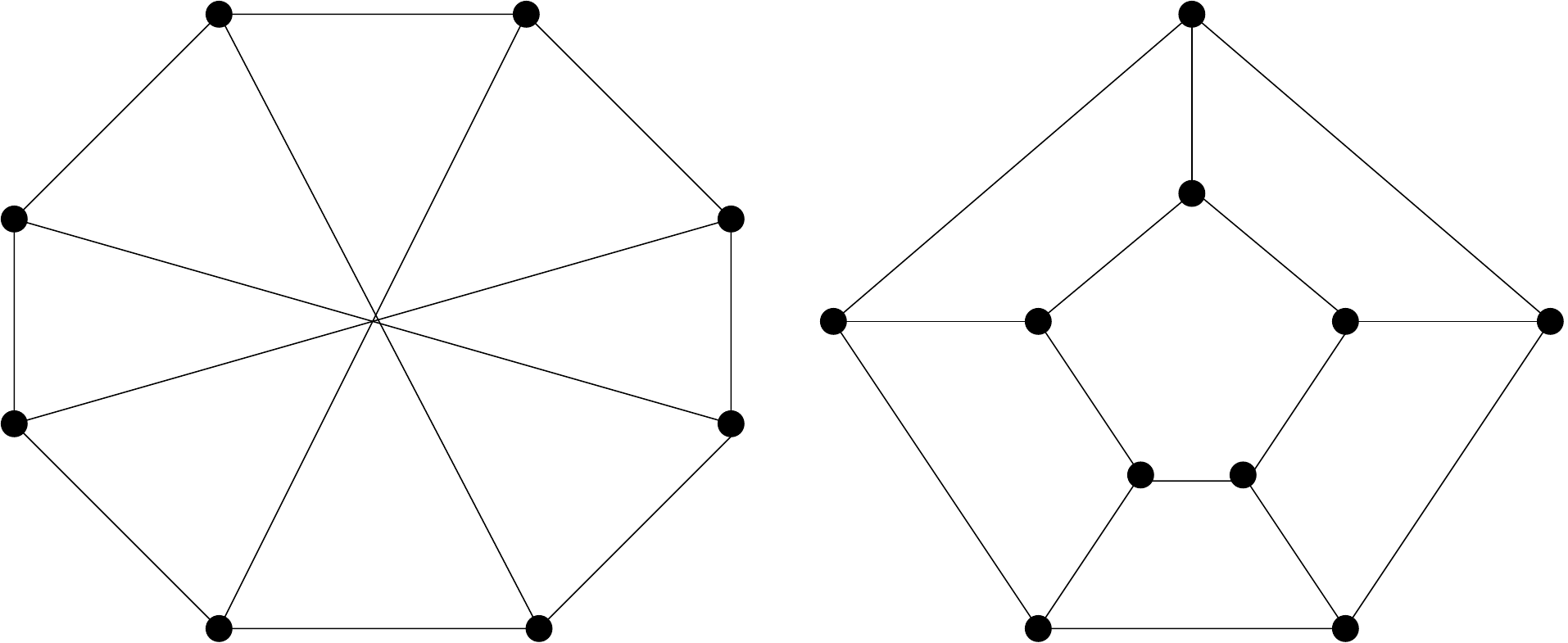}
\caption{The graphs $V_8$ and $C_5\times C_2$.}
\label{v8+c2c5} 
\end{figure}

\begin{theorem}\cite{APC90} \label{theo3tree}
A graph $G$ is a partial 3-tree if and only if $G$ does not have $K_5,K_{2,2,2}, V_8$ and $C_5 \times C_2$ as a minor.
\end{theorem} 
The graphs $V_8$ and $C_5\times C_2$ are shown  in Figure~\ref{v8+c2c5}.
The forbidden minors for partial 3-trees are   natural candidates for being obstructions to the class $\GG_4$. We have already seen  that  for   $K_5$ and $K_{2,2,2}$ this is indeed the case. However, this is not the true for  $V_8$ and $C_5 \times C_2$. Indeed,  in the extended version of the paper, it is proven 
that  $\gd(V_8)=\gd(C_5 \times C_2)=4$~\cite{LV12}. 
Using this, we
can now complete our characterization of the class $\GG_4$.

\begin{theorem}\label{theomain}
For a graph $G$, $\gd(G)\le 4$, if and only if $G$ does not have $K_5$ or $K_{2,2,2}$ as a minor.
\end{theorem}

\begin{proof}
The `only if' part follows from Lemmas \ref{lemminor} and \ref{lemK222}.
The `if part' follows from  the fact that $\gd(V_8)=\gd(C_5 \times C_2)=4$ and  Lemmas \ref{lemminor}, \ref{lemcliquesum}, combined with the
following graph theoretical result, shown in \cite{BC}:
If $G$ is a graph with no $K_5$, $K_{2,2,2}$ minors, then $G$ is a subgraph of a clique sum of copies of $K_4$, $V_8$ and 
$C_5\times C_2$.
\qed
\end{proof}


\section{Using semidefinite programming}\label{boundingviasdp}

In this section we sketch the approach which we will follow in order to bound the Gram dimension of the two graphs $V_8$ and $C_5\times C_2$.

\begin{definition} 
Given a graph $G=(V=[n],E)$,  a configuration of  $G$ is an assignment  of vectors $\p_1,\ldots,\p_n \in \oR^k$  (for some $k \ge 1$)  to the nodes of $G$;
 the pair  $(G,\pb)$ is called a   framework, where we use the notation   $\pb=\{\p_1,\ldots,\p_n\}$. 
Two configurations  $\pb,\bold{q}$ of  $G$ (not necessarily lying  in the same space) are said to be equivalent if $\p_i^T\p_j=q^T_i q_j$ for all $ij \in V\cup E$.

\end{definition}

Our objective  is to show that the two graphs $G=V_8$, $C_5 \times C_2$ belong to  $\GG_4$. That is, we must show that, given any $a\in \SSS_+(G)$, one can construct a Gram representation  $\qb$ of $(G,a)$ lying in the space $\oR^4$.

Along the lines of  \cite{Belk} (which deals with Euclidean distance realizations),  our strategy to achieve this is as follows:
First, we select  an initial Gram representation $\pb$ of $(G,a)$ obtained by `stretching'  as much as possible along a given pair $(i_0,j_0)$ which is not an edge of $G$; more precisely, $\pb$ is a representation of $(G,a)$  which maximizes 
 the inner product $\p_{i_0}^T\p_{j_0}$. As suggested  in \cite{SY06}  (in the context of Euclidean distance realizations), this configuration  $\pb$ can be obtained by solving a semidefinite program; then 
  $\pb$ corresponds to  the Gram representation of an optimal solution $X$ to this program.
  
 In general we cannot yet claim that $\pb$ lies in $\oR^4$. However, we can derive useful information about $\pb$ by  using an optimal solution $\Om$ (which will correspond to a `stress matrix') to the dual semidefinite program. Indeed, the optimality condition $X\Om=0$ will imply some  linear dependencies among the $\p_i$'s  that can be used to show the existence of an equivalent representation $\qb$ of $(G,a)$ in low dimension.
 Roughly speaking, most often,   these dependencies  will force the majority  of the $\p_i$'s to lie in  $\oR^4$, and  one will be able to rotate each remaining vector $\p_j$   about the space  spanned by the vectors labeling  the neighbors of $j$ into $\oR^4$.  Showing that the initial representation $\pb$ can indeed be `folded' into $\oR^4$ as just described makes up the main body of the proof.

\medskip
We now sketch how to model the  `stretching' procedure using semidefinite programming and how to obtain a  `stress matrix' via  semidefinite programming duality.

Let $G=(V=[n],E)$ be a graph and let $e_0=(i_0,j_0)$ be a non-edge of $G$ (i.e., $i_0\ne j_0$ and $e_0\not\in E$).
Let $a\in \SSS_{++}(G)$ be a partial  positive semidefinite matrix for which we want  to show the existence of a Gram representation  in a small dimensional space. 
For this consider  the semidefinite program:
\begin{equation}\label{ex:SDPP}
\max \ \langle E_{i_0j_0},X\rangle\ \ \text{\rm such that  } \langle E_{ij},X\rangle =a_{ij} \ (ij\in V\cup E),\ 
\ X\succeq 0,
\end{equation}
where $E_{ij}=(e_ie_j^T+e_je_i^T)/2$ and  $e_1,\ldots,e_n$ are the standard unit vectors in $\oR^n$.
The     dual semidefinite program  of (\ref{ex:SDPP})  reads:
\begin{equation}\label{ex:SDPD}
\min  \sum_{ij\in V\cup E} w_{ij} a_{ij} 
\text{ such that } \Omega= \sum_{ij\in V\cup E} w_{ij} E_{ij}-E_{i_0j_0}\succeq 0.
\end{equation}
As the program (\ref{ex:SDPD}) is  strictly feasible,
 there is no duality gap and the optimal values are attained in both programs. Consider now a pair $(X,\Omega)$ of primal-dual optimal solutions. Then $(X,\Omega)$ satisfies the optimality condition, i.e., $X\Omega=0$. This condition can be reformulated as
\begin{equation}\label{relequil}
w_{ii}\p_i+\sum_{j|ij \in E \cup \{e_0\}} w_{ij}\p_j=0\ \text{for all } i\in [n],
\end{equation}
where $\Omega=(w_{ij})$ and $X={\rm Gram}(p_1,\ldots,p_n)$. Using the local information provided by the
 `equilibrium' conditions (\ref{relequil}) about the configuration $\pb$   and examining all possible cases for the support of the stress matrix, one can construct an  equivalent configurations in $\oR^4$ for the graphs $V_8$ and $C_5 \times C_2$. 
 
 For the full proof the reader is referred to the extended version of the paper~\cite{LV12}.

\section{Links to Euclidean  graph realizations}\label{seclinks}
In this section we investigate the links between  the notion of Gram dimension and graph realizations in Euclidean spaces which will, in particular, enable us to relate our result from Theorem \ref{theomain} to the result of Belk and Connelly (Theorem  \ref{theoBC}).

Recall that a matrix $D=(d_{ij})\in \SSS^n$ is a {\em Euclidean distance matrix} (EDM) if there exist vectors $p_1,\ldots,p_n\in \oR^k$ (for some $k\ge 1$) such that $d_{ij}=\|p_i-p_j\|^2$ for all $i,j\in [n]$.
Then  $\EDM_n$ denotes the  cone of all $n \times n$ Euclidean distance matrices and, for a graph $G=([n],E)$, $
\EDM(G)=\pi_E(\EDM_n)$ is  the set of partial $G$-matrices that can be completed to a Euclidean distance matrix.

\begin{definition}\label{defedm}
Given a graph $G=([n],E)$ and  $d\in \oR_{+}^{E}$,  a Euclidean (distance) representation of $d$ in $\oR^k$
consists of  a set of vectors $\p_1,\ldots,\p_n\in \oR^k$  such that $$\| \p_i-\p_j\|^2=d_{ij}\  \forall  ij \in E.$$
Then, $\fedm(G,d)$ is the smallest  integer $k$ for which $d$ has a  Euclidean representation in $\oR^k$ and the graph parameter $\fedm(G)$ is defined as 
\begin{equation}\label{edmdimdef}
\fedm(G)=\underset{d \in \EDM(G)}{\max} \fedm(G,d).
\end{equation}
\end{definition}

There is a well known correspondence between psd and EDM completions (for  details and references see, e.g.,  \cite{DL97}).
Namely, for a graph  $G$,  let $\nabla G$ denote its {\em suspension graph},   obtained by adding a new node (the {\em apex} node, denoted by 0), adjacent to all nodes of $G$.
Consider the one-to-one  map $\phi:  \oR^{V \cup E(G)} \mapsto \oR_{+}^{E(\nabla G)}$,  which maps  $x\in \oR^{V \cup E(G)}$ to
 $d=\phi(x )\in \oR_{+}^{E(\nabla G)}$ defined  by
$$ d_{0i}= x_{ii}  \ (i\in [n]),\ \ \ d_{ij}=x_{ii}+x_{jj}-2x_{ij} \ (ij\in E(G)).$$
  Then the vectors $u_1,\ldots,u_n\in\oR^k$ form a Gram representation of $x$ if and only if the vectors  $u_0=0,u_1,\ldots,u_n$ form a Euclidean representation of $d=\phi(x)$ in $\oR^k$. This shows:

\begin{lemma}\label{covariance}
Let $G=(V,E)$ be a graph. Then,  $\gd(G,x)=\fedm(\nabla G,\phi(x))$ for any $x\in \oR^{V\cup E}$ and thus 
 $\gd(G)=\fedm(\nabla G)$.
 \end{lemma}


 For the Gram dimension of a graph one can show the following property:

 \begin{lemma}\label{lem1}
 Consider a graph $G=([n],E)$ and let  $\nabla G=([n]\cup\{0\},E\cup F)$, where $F=\{(0,i)\mid i\in [n]\}$. Given $x\in \oR^E$, its {\em $0$-extension} is the vector $y=(x,0)\in\oR^{E\cup F}$. If $x\in \SSS_+(G)$, then $y\in \SSS_+(\nabla G)$ and
 $\gd(G,x)=\gd(\nabla G, y)$. Moreover, 
 $\gd(\nabla G)= \gd(G)+1$.
 \end{lemma}

 \begin{proof}
 The first part is clear and implies $\gd(\nabla G)\ge \gd(G)+1$. 
 Set $k=\gd(G)$; we  show the reverse inequality $\gd(\nabla G)\le k+1$.
 For this, let $X\in \SSS^{n+1}_+$, written in block-form as $X=\left(\begin{matrix} \alpha & a^T \cr a & A\end{matrix}\right)$,
 where $A\in \psd$ and the first row/column is indexed by the apex node 0 of $\nabla G$.
 If $\alpha =0$ then $a=0$, $\pi_{VE}(A)$ has a Gram representation in $\oR^r$  and thus $\pi_{ V(\nabla G) E(\nabla G)}(X)$ too.
 Assume now $\alpha > 0$ and without loss of generality $\alpha =1$.
 Consider the Schur complement $Y$ of $X$ with respect to the entry $\alpha=1$, given by
  $Y=A-aa^T$.
  As $\gd(G)=k$,  there exists $Z\in \psd$ such that $\text{rank}(Z) \le k$ and $\pi_{V E}(Z)=\pi_{V E}(Y)$.
  Define the matrix $$X':= \left(\begin{matrix} 1 & a^T \cr a & aa^T\end{matrix}\right)+\left(\begin{matrix} 0 & 0 \cr 0 & Z\end{matrix}\right).$$
  Then, ${\rm rank}( X') ={\rm rank} ( Z)+1\le k+1$. Moreover, $X'$ and $X$ coincide at all diagonal entries as well as at all entries corresponding to edges of $\nabla G$. This concludes the proof that $\gd(\nabla G)\le k+1$.
  \qed\end{proof}

  We do not know whether the analogous property is true for the graph parameter $\fedm(G)$. On the other hand, one can prove the following partial result, whose proof was communicated to us by A. Schrijver.

  \begin{theorem}\label{lemedmdim}
  For a graph $G$,
  $\fedm(\nabla G)\ge \fedm(G)+1$.
  \end{theorem}

  \begin{proof}
  Set $\fedm(\nabla G)=k$; we   show  $\fedm(G)\le k-1.$  We may assume that $G$ is connected (else deal with each connected component separately).
  Let  $d \in \EDM(G)$ and let $p_1=0,p_2, \ldots, p_n$ be a Euclidean representation of $d$ in $\oR^m$ ($m\ge 1$).
  Extend the $p_i$'s  to vectors $\widehat{p_i}=(p_i,0)\in \oR^{m+1}$  by appending an extra  coordinate equal to zero, 
   and set $\widehat{p}_0(t)=(0,t)\in \oR^{m+1}$ where $t$ is any positive real  scalar.
    Now consider  the distance   $\widehat{d}(t) \in \EDM(\nabla G)$ with Euclidean representation $\widehat{p_0}(t), \widehat{p_1},\ldots,\widehat{p_n}$.

       As $\fedm(\nabla G)=k$, there exists another Euclidean representation of $\widehat{d}(t)$  by  vectors $q_0(t), q_1(t),\ldots,q_n(t)$ lying  in  $\oR^k$.
         Without loss of generality, we can assume that
	   $q_0(t)=\widehat{p_0}(t)=(0,t)$ and $q_1(t)$ is the zero vector; 
	       for $i\in[n]$, write $q_i(t)=(u_i(t),a_i(t))$, where $u_i(t)\in \oR^{k-1}$ and $a_i(t) \in \oR$.
	           Then $\|q_i(t)\| = \|\widehat{p_i}\|=\|p_i\|$ whenever node $i$ is adjacent to node 1 in $G$.
		        As the graph $G$  is connected,  this implies that, for any $i\in [n]$,
			      the scalars   $\|q_i(t)\|$ ($t \in \oR_+$)  are bounded. Therefore    there exists a  sequence  $t_m \in \oR_+$ ($m\in\oN$)  converging to $+\infty$     and for which the sequence $(q_i(t_m))_m$ has a limit. Say  $q_i(t_m)=(a_i(t_m),u_i(t_m))$ converges to  $(u_i,a_i)\in \oR^k$ as $m \rightarrow +\infty$,  where $u_i\in\oR^{k-1}$ and $a_i\in\oR$.
				      The condition $\|q_0(t)-q_i(t)\|^2=\widehat{d}(t)_{0i}$ implies  that  $\|p_i\|^2+t^2=\|u_i(t)\|^2+(a_i(t)-t)^2$ and thus
				      $$  a_i(t_m)=\frac{a_i^2(t_m)+\|u_i(t_m)\|^2-\|p_i\|^2}{2t_m}  \hspace{0.2cm} \forall m\in \oN.$$
				       Taking the limit as $m \to \infty $ we obtain  that $\underset{m \to \infty }{\lim} a_i(t_m)=0$ and thus  $a_i=0$.
				        Then, for $i,j\in [n]$, $d_{ij}=\widehat{d}(t_m)_{ij}=\|(a_i(t_m),u_i(t_m))-(a_j(t_m),u_j(t_m))\|^2$ and taking the limit as $m \to +\infty$ we obtain that  $d_{ij}=\|u_i-u_j\|^2$.
					 This shows that  the vectors $u_1,\ldots,u_n$ form a Euclidean representation of $d$ in $\oR^{k-1}$.
					 \qed\end{proof}

This raises the following question:
Is it true that
$\fedm(\nabla G)\le \fedm(G)+1$?
A positive answer  would imply that our characterization for the graphs with Gram dimension 4 (Theorem \ref{theomain}) is equivalent to the characterization of Belk and Connelly for 
the graphs having a Euclidean representation in $\oR^3$ (Theorem \ref{theoBC}).
In any case,  we have that: 
\begin{equation}\label{in:gd>ed}
\gd(G)=\fedm(\nabla G)\ge \fedm(G)+1.
\end{equation}
In the full version of the paper it is  proven that  $\gd(V_8)=\gd(C_5\times C_2)=4$~\cite{LV12}. This fact combined with~(\ref{in:gd>ed}) implies that $\fedm(V_8)=\fedm(C_2\times C_5)=3$, which was the main part  in the  proof of Belk \cite{Belk} to characterize graphs with $\fedm(G)\le 3$. 

\section{Some complexity  results}\label{sec:complexity}

Consider the natural  decision problem associated with the graph parameter $\gd(G)$: 
Given a   graph $G$ and a rational vector $x \in \EE(G)$,
determine whether $\gd(G,x)\le k$, where $k\ge 1$ is some  fixed  integer.
In this section we show  that this is a hard problem for any $k\ge 3$, already when $x$ is the zero vector. 
Further results concerning complexity issues associated with the graph parameter $\gd(G)$ are discussed   in~\cite{NLV12}.

Recall that an orthogonal representation of dimension $k$ of $G=([n],E)$   is a set of nonzero vectors $v_1,\ldots,v_n\in \oR^k$ such that $v_i^Tv_j=0$ for all pairs $ij\not \in E.$
Clearly, the minimum dimension of an orthogonal representation of the complementary graph $\overline{G}$ coincides with 
     $\gd(\oG,0)$; this graph parameter is called the  {\em orthogonality dimension} of $G$,
  also denoted by $\xi(G)$.
  Note that  it  satisfies the inequalities 
$ \omega (G)\le \xi(G) \le \chi( G)$,    where $\omega(G)$ and $\chi(G)$ are the clique and chromatic numbers of $G$ (see
\cite{Lo79}). 

    One can easily verify that, for $k=1,2$,
    $\xi(G) \le k $ if and only if $\chi(G)\le k$, which can thus  be  tested in polynomial time.
    On the other hand, for  $k=3$, Peeters \cite{Pe96} gives a polynomial time reduction of the problem of testing $\xi(G)\le 3$ to the problem of testing $\chi(G)\le 3$; moreover this reduction preserves graph planarity. As a consequence, it is NP-hard to check whether $\gd(G,0) \le 3$, already  for the class of planar graphs.

	   This hardness result for the zero vector  extends to any $k\ge 3$, using the operation of adding an apex node to a graph.
	   For a graph $G$, 
	   $\nabla ^kG$ is the new graph obtained by adding  iteratively $k$ apex nodes to $G$.


  \begin{theorem}
	   For any fixed $k\ge 3$, it is NP-hard to decide whether $\gd(G,0)\le k$, already for graphs $G$ of the form $G=\nabla ^{k-3} H$ where $H$ is planar.
	   \end{theorem}

\begin{proof} 
Use the result of Peeters \cite{Pe96} for $k=3$, combined with the first part of Lemma \ref{lem1} for $k\ge 4$.
\qed\end{proof}

Combining with Lemma \ref{covariance} this implies that, for any fixed $k\ge 3$, 
 it is NP-hard to decide whether $\text{\rm ed}(G,d)\le k$, already when $G=\nabla^{k-2}H$ where $H$ is planar and $d\in \{1,2\}^E$.
 In comparison,  Saxe \cite{Sa79} showed NP-hardness for  any $k\ge 1$ and for $d\in \{1,2\}^E$.

\medskip\noindent
{\bf Acknowledgements.} We thank  M. E.-Nagy for useful discussions and A. Schrijver for his suggestions for the proof of  Theorem  \ref{lemedmdim}.

\end{document}